\documentclass[11pt]{article}
\usepackage{sychung2023}

\title{On the generalized interlacing property for the zeros of Bessel functions}

\author{Seok-Young Chung\footnote{Seok-Young.Chung@ucf.edu. Department of Mathematics, University of Central Florida,
4393 Andromeda Loop N., Orlando, FL 32816, USA.}
\and Sujin Lee\footnote{watertrue015@gmail.com. Department of Mathematics, College of Natural Sciences, Chung-Ang University,
84 Heukseok-Ro, Dongjak-Gu, Seoul 06974, Korea.}
\and Young Woong Park\footnote{ywpark1839@gmail.com. Department of Mathematics, College of Natural Sciences, Chung-Ang University,
84 Heukseok-Ro, Dongjak-Gu, Seoul 06974, Korea.}}
\date{}

\begin{document}

\maketitle

{\bf Abstract.}
This paper investigates a generalized interlacing property between Bessel functions, particularly $J_\nu$ and $J_\mu$, where the difference $\vb{\nu-\mu}$ exceeds $2$. This interlacing phenomenon is marked by a compensatory interaction with the zeros of Lommel polynomials, extending our understanding beyond the traditional $\vb{\nu-\mu} \le 2$ regime. The paper extends the generalized interlacing property to cylinder functions and derivative of Bessel functions, as an application. It is also discussed that Siegel's extension of Bourget hypothesis to rational numbers of $\nu$ cannot be further improved to arbitrary real numbers.

\medskip


\medskip


\section{Introduction}

The primary focus of this paper is to investigate a generalization of the interlacing property between Bessel functions $J_\nu$ and $J_\mu$ when the difference $\vb{\nu-\mu}$ exceeds $2$. The key objective is to examine the interlacing behavior between the zeros of $J_\nu$ and $J_{\nu+m}$ for $m=3,\,4,\,\ldots\,$, where this interlacing phenomenon is characterized by the compensatory interaction with the zeros of the Lommel polynomials. By establishing the interplay between these functions, we aim to provide a comprehensive understanding of the extended interlacing property beyond the conventional regime of $\vb{\nu-\mu} \le 2$. 

As customary, $J_\nu$ stands for the first kind Bessel function of order $\nu$, defined as
\begin{equation*}
    J_\nu(x) = \sum_{k=0}^\infty \frac{ (-1)^k }{k! \,\Gamma(\nu+k+1)}\rb{\frac{x}{2}}^{2k+\nu},
\end{equation*}
where $\Gamma$ denotes the gamma function. We denote by $\cb{j_{\lambda,k}}_{k=1}^\infty$ the sequence of all positive zeros of the Bessel function $J_\lambda$ arranged in ascending order of magnitude. It is important to note that unless explicitly stated otherwise, all sequences of positive zeros in this paper are arranged in ascending order of magnitude. We might abbreviate sub-indices of a zero if they are obvious in the context.

To offer further insight into the interlacing property of functions $J_\nu$ and $J_{\nu+m}$, let us revisit the patterns observed in relation to the patterns (see \cite[\S 15.22]{Watson2} and \cite{CC1}) with respect to the positive zeros of the first kind Bessel functions $J_\nu$ and $J_\mu$, of order $\nu$ and $\mu=\nu+1$ or $\nu+2$ respectively, as follows:
\begin{equation}\label{pattern}
    \begin{split}
        0< j_{\nu,1} < j_{\mu,1} < j_{\nu,2} < j_{\mu,2} < \cdots, \quad (\mu>0) \\
        0< j_{\mu,1} < j_{\nu,1} < j_{\mu,2} < j_{\nu,2} < \cdots, \quad (\mu\le 0)
    \end{split}
\end{equation}
in which the case of $\mu =\nu +2$ and $\nu = - 1$ should be excluded.
It is said to have \emph{interlacing property} if the positive zeros of $J_\nu$ and $J_\mu$ satisfy one of the patterns described in \eqref{pattern}. Furthermore, Schl{\"a}fli's classical formula establishes a relationship between the derivative of a specific positive zero $j$ of $J_\nu$, where $\nu>0$, with respect to $\nu$ and an integral involving Bessel functions
\begin{equation*}
\frac{dj}{d\nu} = \frac{2\nu}{j J_{\nu+1}^2(j)} \int_0^j J_\nu^2(t) \frac{dt}{t}.
\end{equation*}

It is also noteworthy to point out two alternative forms, as presented by Ifantis, Siafarikas \cite{Ifantis} and Ismail, Muldoon \cite{Ismail2}, both of which provide a discrete version of representation for $dj/d\nu$ in terms of a series of interconnected functions
\begin{align*}
    \frac{dj}{d\nu} &= j\rb{ \sum_{k=1}^\infty J_{\nu+k}^2(j) } \bigg/ \rb{ \sum_{k=1}^\infty \rb{\nu+k} J_{\nu+k}^2(j) }\\
    &= \frac{2}{j} \sum_{k=0}^\infty R_{k,\nu+1}^2\rb{j}.
\end{align*}

Additionally, a well-known result due to Watson provides a formula for the derivative of a specific positive zero $c$ of the cylinder function $C_\nu^\alpha= \cos{\alpha} J_\nu - \sin{\alpha} Y_\nu$, for any real values of $\nu$ and $0\leq\alpha<\pi$,
\begin{equation}\label{W}
    \frac{dc}{d\nu} = 2c \int_0^\infty K_0\rb{2c \sinh{t}} e^{-2\nu t}dt>0,
\end{equation}
where $Y_\nu$ stands for the second kind Bessel function of order $\nu$ and $K_0$ stands for the second kind modified Bessel function of order $0$. It is advised to refer to \cite[\S 15.6]{Watson2} and \cite{Ismail2} for further details and derivations of the aforementioned formulas.

Considering these formulas, it is evident that the positive zero $j$ of $J_\nu$ is a monotonically increasing function with respect to $\nu$. Moreover, as it turns out, the positive zero $j_{\nu,k}$, $k\ge1$, tends to infinity as $\nu$ approaches infinity. (see \cite{Kerimov}) Consequently, as $\mu$ becomes large, the positive zero $j_\mu$ of $J_\mu$ moves monotonically farther away from the origin, leading to the inevitable breakdown of the interlacing property between $J_\nu$ and $J_\mu$ in the sense that
\begin{equation*}
    j_{\nu,1} < j_{\nu,2} < j_{\mu,1}, \quad (\nu \ll \mu).
\end{equation*}

The result established by P{\'a}lmai \cite{Palmai} regarding the interlacing property between Bessel functions $J_\nu$ and $J_\mu$ can be concisely stated as follows:
\begin{varthm}{A}\label{Theorem A}
For $\nu>-1$ and $\mu>-1$, the positive zeros of $J_\nu$ and $J_\mu$ are interlaced if and only if $0<\vb{\nu-\mu}\le 2$.
\end{varthm}
\noindent The theorem presented above describes the necessary and sufficient condition for when the interlacing property breaks down. We refer to the survey papers \cite{Elbert}, \cite{Laforgia} for further details and additional information regarding the zeros of Bessel functions.

Optimal parameter intervals, particularly in the context of interlacing properties, have been a subject of investigation for a considerable duration, not limited to Bessel functions. Driver, for instance, conducted research into different orthogonal polynomials, such as Laguerre and Jacobi polynomials, with the aim of identifying these optimal intervals. (see for instance \cite{Driver1} and \cite{Driver3}) It would be of interest to explore the applicability of the generalized interlacing property within these specified intervals.

In this comprehensive study, we investigate the intricate relationship between Bessel functions and Lommel polynomials, shedding light on their interplay through common zeros. In section 2, we focus our attention on the examination of the common zeros of $J_\nu$ and $J_{\nu+m}$, thereby laying the foundation for a deeper understanding of their behavior. In particular, we show that Bourget's hypothesis cannot be extended to arbitrary real values of a parameter $\nu$. Section 3 marks a pivotal point as we unveil the core of our work; the main theorem on the generalized interlacing property, characterized by the compensatory interaction with the zeros of the Lommel polynomials, in conjunction with the common zeros we discussed. In section 4 and 5, we extend our results to the realm of cylinder functions and derivative of the Bessel functions, broadening the scope of our interlacing principle. 

As for the requisite background, we present the definition of the Lommel polynomials $R_{m,\nu}$ (see \cite{Erdelyi}), which arises as a result of an iterative recurrence for the Bessel functions
\begin{equation}\label{BL1}
    J_{\nu+m}(x) - R_{m,\nu}(x) J_\nu(x) + R_{m-1,\nu+1}(x) J_{\nu-1}(x) = 0,
\end{equation}
and these are polynomials in $1/x$, expressed explicitly as
\begin{equation}\label{L}
    R_{m,\nu}(x) = \sum_{k=0}^{ \lfloor \frac{m}{2} \rfloor } (-1)^k \binom{m-k}{k} \frac{ \Gamma\rb{ \nu+m-k } }{ \Gamma\rb{\nu+k} } \rb{\frac{x}{2}}^{2k-m},
\end{equation}
for $\nu,x\in \R$ and $m=0,\,1,\,\ldots\,$, where $\lfloor a \rfloor$ stands for the floor function of $a\in \R$. An extension with respect to the negative parameter $m$ can be derived using the Graf's formula
\begin{equation}\label{Graf}
    R_{-m,\nu}(x) = -R_{m-2,\nu-m+1}(x), \for m\in \mathbb{Z},
\end{equation}
Indeed, the Lommel polynomials satisfy the following recurrence relation
\begin{equation}\label{Lommel1}
    R_{m-1,\nu}(x) + R_{m+1,\nu}(x) = \frac{2\rb{\nu+m}}{x}R_{m,\nu}(x).
\end{equation}
We denote by $\{\rho_{n,\lambda,k}\}_{k\ge 1}$ the sequence of all positive zeros of $R_{n,\lambda+1}$.

\section{The zeros of \texorpdfstring{$J_\nu$}{Jnu} and \texorpdfstring{$J_{\nu+m}$}{Jnu+m} in common}
The objective of this section is to delve into the characteristics of the common zeros of $J_\nu$ and $J_{\nu+m}$. We introduce the notation $\mathcal{N}_{m,\nu}$ for $\nu\in\R$ and $m\in \mathbb{N}$ as
\begin{equation}\label{Z1}
    \mathcal{N}_{m,\nu} = \cb{ j_{\nu,k}\,:\, J_{\nu+m}\rb{j_{\nu,k}}=0 },
\end{equation}
which denotes the collection of positive zeros common to $J_\nu$ and $J_{\nu+m}$.

An essential observation to highlight is the interconnection between the set $\mathcal{N}_{m,\nu}$ and the Lommel polynomial $R_{m-1,\nu+1}$. A simple consequence from the equation \eqref{L} can be illustrated as
\begin{observation}\label{ob1}
Let $\nu>-1$ and $m\in \mathbb{N}$. Then we have
\begin{equation*}
    \mathcal{N}_{m,\nu} = \cb{ j_{\nu,k}\,:\, R_{m-1,\nu+1}\rb{j_{\nu,k}}=0 }.
\end{equation*}
Consequently,
\begin{equation*}
    \vb{ \mathcal{N}_{m,\nu} } \le \left\lfloor \frac{m-1}{2} \right\rfloor.
\end{equation*}
\end{observation}

\begin{proof}
We suppose $\zeta$ is a common zero of $J_\nu$ and $J_{\nu+m}$. By examining equation \eqref{BL1}, we deduce that $R_{m-1,\nu+1}(\zeta)J_{\nu-1}(\zeta)=0$. Since $J_\nu$ and $J_{\nu-1}$ have no common zeros, it follows that $\zeta$ must also be a zero of $R_{m-1,\nu+1}$. On the other hand, if we assume $\zeta$ is a common zero of $J_\nu$ and $R_{m-1,\nu+1}$, then it follows from equation \eqref{BL1} that $J_{\nu+m}(\zeta)=0$. Hence the result is immediate.

Consequently, the cardinality of $\mathcal{N}_{m,\nu}$ cannot exceed the number of possible positive zeros of $R_{m-1,\nu+1}$, which is $\left\lfloor \frac{m-1}{2} \right\rfloor$ by virtue of the definition \eqref{L}.
\end{proof}

We next spotlight Bourget's hypothesis, introduced in the 19th century, which states that the first kind Bessel functions $J_\nu$ and $J_{\nu+m}$ has no common zeros apart from the origin. This conjecture postulates for all nonnegative integral values of $\nu$ and positive integral values of $m$. For historical background and related results, one can refer to \cite[\S 15]{Watson2}. In 1929, Siegel's seminal work \cite[\S 4]{Siegel} established a more comprehensive theorem that led to the verification of Bourget's hypothesis as a straightforward consequence, as stated below
\begin{varthm}{B}\label{B}
    Let $m$ be a positive integer, and let $\nu$ be a rational number such that $-2\nu \ne m$. Then $J_{\nu}$ and $J_{\nu+m}$ have no common zeros except for the origin.
\end{varthm}

Additionally, Petropoulou et al. \cite{Petro} tackle the issue of the common zeros between two functions, namely, $J_\nu$ and $J_\mu$, when the difference between $\nu$ and $\mu$ is not a positive integer.

To explicate the case when $-2\nu = m$, we turn our attention to the classical Wronskian formula, given by
\begin{equation*}
    W\qb{J_\nu,\, J_{\nu+m}}(x) = W\qb{J_{\nu},\,J_{-\nu}}(x) = -\frac{2\sin\rb{\nu \pi}}{\pi x},
\end{equation*}
which keeps constant sign for $x>0$ and fixed $\nu\notin \mathbb{Z}$. Consequently, when $\nu$ is not an integer, $J_\nu$ and $J_{\nu+m}$ cannot share common zeros, as the Wronskian of these functions must vanish at any point where they have common zeros.  This fact, in conjunction with Theorem \ref{B}, leads to the following observation.
\begin{observation}\label{ob2}
    Let $m\in \mathbb{N}$, and let $\nu$ be a rational number such that $-2\nu \ne m$ if $\nu \in\mathbb{Z}$. Then
    $\mathcal{N}_{m,\nu}$ must be empty.
\end{observation}

One may consider extending Theorem \ref{B} for irrational values of $\nu$. Nevertheless, such an extension is not feasible. We begin by introducing a significant lemma.

\begin{lemma} \label{BLem1}
Let $n\ge 2$ be fixed, and let $\rho=\rho(\nu)$ be a zero of the Lommel polynomial $R_{n,\nu+1}.$ Then
   \begin{equation*}
        1 <\lim_{\nu \to\infty}\frac{d\rho}{d\nu} < \infty .
    \end{equation*}
\end{lemma}

\begin{proof}
We initiate from the property of a zero $\rho$ of the Lommel polynomial $R_{n,\nu+1}$ that 
\begin{equation}\label{order}
    \rho > j_{\nu, 1} \for \nu>-1,
\end{equation}
which can be found in \cite[p. 375]{Grosjean}. Moreover, Ismail \cite[Theorem 4 and 5]{Ismail1} proved that, for $\nu>-1$, $\rho$ is differentiable and the function $\rho(\nu)/(\nu+1)$ is decreasing. On account of these properties and the asymptotic behavior of the zeros of Bessel functions (see \cite[\S 8]{Olver}), stated that $j_{\nu,k} = \nu + O(\nu^{1/3})$ as $\nu \to \infty$, we conclude that $\lim_{\nu \to\infty}\frac{d\rho}{d\nu}$ exists and satisfies
\begin{equation*}
    \lim_{\nu \to\infty}\frac{d\rho}{d\nu} = \lim_{\nu \to \infty} \frac{\rho(\nu)}{\nu+1} \ge \lim_{\nu \to \infty} \frac{j_{\nu,1}}{\nu+1} =1.
\end{equation*}

We now suppose, on the contrary, that $\eta :=\lim_{\nu \to \infty}\frac{d\rho}{d\nu} = 1$. Then it follows from the equation \eqref{Lommel1} with $x=\rho$ that
\begin{equation}\label{lem1}
    1 = \lim_{\nu \to \infty} \frac{\rho(\nu)}{\nu+n} = 2 \lim_{\nu \to \infty} \frac{R_{n-1,\nu+1}(\rho)}{R_{n-2,\nu+1}(\rho)}.
\end{equation}
On the other hand, by the definition of $R_{m,\nu+1}$ in hypergeometric series (see \cite[\S 9.62 (5)]{Watson2}) and Chu-Vandermonde identity, we find that for each $m\in \mathbb{N}$,
\begin{align}
    \lim_{\nu \to \infty} {R_{m,\nu+1}(\rho)}/{2^m} &=   {}_2F_1\left[\begin{array}{c} (1-m)/2,\, -m/2 \\
    -m \end{array}\biggr| 1 \right] \notag \\
    &= \begin{cases}
    \frac{(-(m+1)/2)_{m/2}}{(-m)_{m/2}}, & \ifz m\text{ is even},\\
    \frac{(-m/2)_{(m-1)/2}}{(-m)_{(m-1)/2}}, & \ifz m\text{ is odd}. \label{lem2}
    \end{cases}
\end{align}
We note that $R_{0,\nu+1} \equiv 1$ and $(a)_n$ denotes the Pochhammer's symbol, defined as $(a)_n = a(a+1)\cdots(a+n-1)$, $n \in \mathbb{N}$ and $\rb{a}_0=1$. Hence it can be deduced from \eqref{lem1} and \eqref{lem2} that
\begin{equation*}
    1 = 2 \lim_{\nu \to \infty} \frac{R_{n-1,\nu+1}(\rho)}{R_{n-2,\nu+1}(\rho)} =  \frac{2n}{n-1},
\end{equation*}
which is a contradiction for $n\ge2$, and therefore completes the proof.
\end{proof}

\begin{remark}
    Let us define $\eta = \lim_{\nu \to \infty}\frac{d\rho}{d\nu}$. With this definition, the equation \eqref{lem1} can be reformulated as
    \begin{equation*}
        {}_2F_1\left[\begin{array}{c} (1-m)/2,\, -m/2 \\
        -m \end{array}\biggr| \,\eta^2 \right] = 0.
    \end{equation*}
    Notably, Driver and M{\"o}ller \cite[Theorem 3.1 and Colloary 3.2]{Driver2} investigated the zeros of hypergeometric polynomials ${}_2F_1(-n,\, b;\, -2n;\, x)$. These results offer an alternative proof for Lemma \ref{BLem1}, leading that $\eta$ must be on $\rb{1,\infty}$.
\end{remark}

We shall prove that Theorem \ref{B} cannot be improved for arbitrary irrational values of $\nu$. 
\begin{theorem}\label{thm:2.1}
    For each $m \geq 3$, there exists $\nu^*\in (-1,\infty) \setminus \mathbb{Q}$ such that $\mathcal{N}_{m,\nu^*} \ne \emptyset$, \rm{i.e.},
    $J_{\nu^*}$ and $J_{\nu^*+m}$ share at least one common zero.
\end{theorem}

\begin{proof}
Let $m \ge 3$, $\ell\ge1$ be fixed. We consider the $\ell$-th zero $\rho_{m,\nu,\ell}$ of the Lommel polynomial $R_{m,\nu+1}$. 
It is well-known facts that $j_{\nu,k}$ is a continuously differentiable function with respect to $\nu$, and it tends towards infinity as $k \to \infty$.
Given a fixed $\nu_0>-1$, we choose a specific $k$ such that $\rho_{m-1,\nu_0,\ell}< j_{\nu_0,k}$ and define $d(\nu) := \rho_{m-1,\nu,\ell} - j_{\nu,k}$. This function is continuously differentiable with respect to $\nu$ and has the property that $d(\nu_0)<0$. 
Lemma \ref{BLem1} implies that $d(\nu)/\nu > 0$ as $\nu \to \infty$. Consequently, we conclude that $d(\nu)>0$ for sufficiently large $\nu$. This leads to the existence of $\nu^*\in \rb{-1,\infty}$ such that $\rho_{m-1,\nu^*,\ell} = j_{\nu^* ,k}$ which yields the desired result by considering that $\nu^*$ is irrational due to Theorem \ref{B}.
\end{proof}

\begin{remark}
Indeed, on inspecting the proof of Theorem \ref{thm:2.1}, we could find countable number of $\nu^*$ for each $m\ge3$. The result of Theorem \ref{thm:2.1} can also be extended to the cylinder function $C_{\nu}^\alpha$.
\end{remark}

\begin{figure}[!ht]
    \centering
    \includegraphics[width=0.75\textwidth]{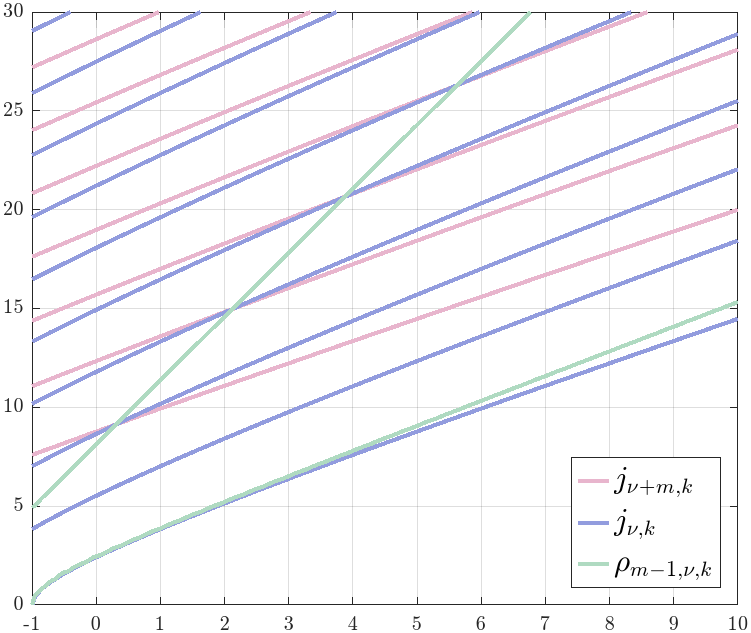}
    \caption{The trajectories of the positive zeros in the $(\nu,x)$-plane in the case of $m=5$, depicting how they collectively approach and share common zeros.}
\end{figure}

\section{Interlacing property between the Bessel functions and the Lommel polynomials}

In this section, we investigate the interlacing property between Bessel functions and Lommel polynomials, elucidating their interaction based on common zeros. To begin our analysis, we establish a Wronskian formula between $J_\nu$ and $R_{m-1,\nu+1}J_{\nu+m}$.
\begin{proposition}\label{Wron1}
For $\nu>-1$ and $m\in \mathbb{N}$, we have
\begin{align*}
    W\qb{J_\nu,\, R_{m-1,\nu+1}J_{\nu+m}}(x)=2 J_{\nu+m}^2(x) \left[ R_{m-1,\nu+1}^2(x)  \sum_{k=1}^\infty \frac{x^2+j_{\nu+m, k}^2}{\big( x^2 - j_{\nu+m, k}^2 \big)^2  } \right. \hspace{11.8pt}\, \\
    \pushright{  \left. + \frac{1}{x^2} \sum_{k=0}^{m-1} \rb{\nu+k+1} R_{k,\nu+1}^2(x)\right].}
\end{align*}
\end{proposition}

\begin{proof}
In view of the Graf's formula \eqref{Graf}, the equation \eqref{BL1} can be reformulated as
\begin{equation}\label{BL2}
    J_{\nu}(x) = -R_{m-2,\nu+1}(x) J_{\nu+m}(x) +  R_{m-1,\nu+1}(x) J_{\nu+m-1}(x).
\end{equation}
We rewrite the Wronskian $W\qb{J_\nu,\, R_{m-1,\nu+1}J_{\nu+m}}(x)$ as
\begin{equation*}
    \frac{W\qb{J_\nu,\, R_{m-1,\nu+1}J_{\nu+m}}(x)}{R_{m-1,\nu+1}^2(x)J_{\nu+m}^2(x)} = - \frac{d}{dx}\frac{J_\nu(x)}{R_{m-1,\nu+1}(x)J_{\nu+m}(x)},
\end{equation*}
which immediately gives, by substituting the right-hand side of the equation \eqref{BL2} for $J_\nu$,
\begin{equation}\label{BL3}
    \begin{split}
        W\qb{ J_\nu,\, R_{m-1,\nu+1}J_{\nu+m} }(x)
        = -J_{\nu+m}^2(x) W\qb{ R_{m-2,\nu+1}, R_{m-1,\nu+1} }(x)\\
        + R_{m-1,\nu+1}^2(x) W\qb{ J_{\nu+m-1}, J_{\nu+m} }(x).
    \end{split}
\end{equation}

Regarding an expansion formula for the Wronskian $W\qb{ J_{\nu+m-1}, J_{\nu+m} }$, we invoke a result \cite{CCP} regarding the partial fraction expansion of the Bessel functions, which states the following:
\begin{equation}\label{partial1}
    \mathbb{J}_\lambda(x) =\mathbb{J}_\mu(x)\qb{ 1 + x
    \sum_{k=1}^\infty \frac{\mathbb{J}_\lambda\rb{j_{\mu, k}}}{j_{\mu, k}\, \mathbb{J}_{\mu}'\rb{j_{\mu, k}}} \rb{ \frac{1}{x-j_{\mu,k}} + \frac{1}{x+j_{\mu,k}}}  },
\end{equation}
valid in the range $\lambda>-1\,$ and $-1<\mu<\lambda+2$, where $\mathbb{J}_\nu$ is defined as
\begin{equation*}
    \mathbb{J}_\nu(x) = \Gamma\rb{\nu+1}\rb{\frac{x}{2}}^{-\nu} J_\nu(x).
\end{equation*}
If we take $\mu=\nu+1$, $\lambda = \nu$, the expansion \eqref{partial1} reduces to
\begin{equation}\label{partial2}
    \frac{\mathbb{J}_{\nu}(x)}{\mathbb{J}_{\nu+1}(x)} = 1 + x
    \sum_{k=1}^\infty \frac{ \mathbb{J}_\nu\rb{j_{\nu+1, k}} }{j_{{\nu+1}, k}\, \mathbb{J}_{\nu+1}'\rb{j_{\nu+1, k}} } \rb{ \frac{1}{x-j_{{\nu+1},k}} + \frac{1}{x+j_{{\nu+1},k}}}.
\end{equation}
By making use of the expansion \eqref{partial2} and the following identities
\begin{align*}
    \mathbb{J}_\nu\rb{j_{\nu+1, k}} &= \frac{j_{\nu+1, k}}{2\rb{\nu+1}}  \mathbb{J}_{\nu+1}'\rb{j_{\nu+1, k}}, \\
    W\qb{ J_{\nu+m-1}, J_{\nu+m} }(x) &= - J_{\nu+m}^2(x) \frac{d}{dx} \frac{J_{\nu+m-1} (x)}{J_{\nu+m}(x)} \\
    &= -2\rb{\nu+m} J_{\nu+m}^2(x) \frac{d}{dx} \rb{\frac{\mathbb{J}_{\nu+m-1}(x)}{x\,\mathbb{J}_{\nu+m}(x)}},
\end{align*}
which is readily verified, it is straightforward to see that
\begin{equation}\label{BL4}
    W\qb{ J_{\nu+m-1}, J_{\nu+m} }(x) =J_{\nu+m}^2(x) \qb{ \frac{2\rb{\nu+m}}{x^2} + 2\sum_{k=1}^\infty \frac{x^2+j_{\nu+m, k}^2}{\big( x^2 - j_{\nu+m, k}^2 \big)^2  } }.
\end{equation}

To deal with the Wronskian $W\qb{ R_{m-2,\nu+1}, R_{m-1,\nu+1} }(x)$, we next turn to analyzing the recurrence of Lommel polynomials \eqref{Lommel1} in accordance of parameters with $m \mapsto m-2$ and $\nu \mapsto \nu+1$, given by
\begin{equation*}
    R_{m-3,\nu+1}(x) + R_{m-1,\nu+1}(x) = \frac{2\rb{\nu+m-1}}{x}R_{m-2,\nu+1}(x).
\end{equation*}
Dividing both sides by $R_{m-2,\nu+1}$ and differentiating, it is a simple matter to write
\begin{multline*}
    W\qb{ R_{m-2,\nu+1}, R_{m-1,\nu+1} }(x) = W\qb{ R_{m-3,\nu+1}, R_{m-2,\nu+1} }(x)\\
    - \frac{2\rb{\nu+m-1}}{x^2} R_{m-2,\nu+1}^2(x).
\end{multline*}
Repeated application of the above, along with the fact that $R_{-1,\nu} \equiv 0$, enables us to find
\begin{equation}\label{BL5}
    W\qb{ R_{m-2,\nu+1}, R_{m-1,\nu+1} }(x)
    = - \frac{2}{x^2} \sum_{k=0}^{m-2} \rb{\nu+k+1} R_{k,\nu+1}^2(x).
\end{equation}
Combining \eqref{BL3}, \eqref{BL4} and \eqref{BL5}, the proof is complete.
\end{proof}

For $\nu>-1,$ $R_{m-1,\nu+1}$ and $R_{m-2,\nu+1}$ do not share any common zeros. If they did, according to the recurrence \eqref{Lommel1}, we find that $R_{k,\nu+1}$ vanishes simultaneously for all $k\in \mathbb{Z}$ at these common zeros, which contradict to $R_{0,\nu+1}\equiv 1$.
Hence the non-negativity of the aforementioned Wronskian is apparent. Furthermore, for $x>0$, it vanishes at $x\in \mathcal{N}$ and
\begin{equation*}
    W\qb{J_\nu,\, R_{m-1,\nu+1}J_{\nu+m}}(x)>0 \quad x \notin \mathcal{N}.
\end{equation*}
When $\mathcal{N}$ is an empty set, the Wronskian remains positive.

To facilitate our analysis, we introduce some important definitions.
\begin{definition}\label{def:3.1}
Let 
\begin{align*}
    \cb{ {j}_{\nu,k_\ell} }_{\ell=1}^{\infty} &= \cb{ j_{\nu,k} }_{k=1}^{\infty} \setminus \mathcal{N}_{m,\nu},  \\
    \cb{ \omega_{m,\nu,k} }_{k=1}^{\infty} &= \cb{j_{\nu+m,k}}_{k=1}^\infty \cup \cb{ \rho_{m-1,\nu,k} }_{k\ge 1},
\end{align*}  
where those are arranged in ascending order of magnitude.
\end{definition}
We present one of our main theorems as follows:
\begin{theorem}\label{thm:3.1}
For $\nu > -1$ and $m\in \mathbb{N}$, if $\mathcal{N}_{m,\nu}=\emptyset$, the positive zeros of $J_\nu$ and $R_{m-1,\nu+1}J_{\nu+m}$ are interlaced according to the following pattern
\begin{equation*}
    0 < j_{\nu,1} < \omega_{m,\nu,1} < j_{\nu,2} < \omega_{m,\nu,2} < \cdots.
\end{equation*}
\end{theorem}

\begin{proof}
For the sake of simplicity, we define
\begin{equation*}
    \phi_{m,\nu}(x) = R_{m-1,\nu+1}(x) J_{\nu+m}(x).
\end{equation*}
On inspecting the Wronskian formula in Proposition \ref{Wron1}, the Wronskian $W\qb{ J_{\nu},\, \phi_{m,\nu} }$ is readily seen to be an entire function with removable singularities at $x= \pm j_{\nu+m,k}$, $k=0,\,1,\,2,\,\ldots\,$. Under the assumption $\mathcal{N}_{m,\nu}=\emptyset$, it is strictly positive for $\nu>-1$, $m\in \mathbb{N}$ and $x \in \R$. Consequently, the zeros of $\phi_{m,\nu}$ are all simple. This is due to the necessity of the Wronskian to vanish at the multiple zeros of $\phi_{m,\nu}$.

We begin with proving the existence of the zeros of $\phi_{m,\nu}$ on each interval partitioned with the zeros of $J_{\nu}$.
Let $j$ and $\bar{j}$ be any two consecutive positive zeros of $J_{\nu}$ with $\bar{j}>j$. It turns out that the zeros of $J_\nu$ are all simple, and thus it is plain to see
\begin{equation*}
    J_{\nu}'\rb{j}J_{\nu}'\rb{\bar{j}} <0.
\end{equation*}
Since $W\qb{J_{\nu},\phi_{m,\nu}}(x) >0$ for all $x>0$, we find that
\begin{equation*}
    \phi_{m,\nu}\rb{j} \phi_{m,\nu}\rb{\bar{j}} = \frac{W\qb{J_{\nu},\phi_{m,\nu}}\rb{j} W\qb{J_{\nu},\phi_{m,\nu}}\rb{\bar{j}} }{ J_{\nu}'\rb{j} J_{\nu}'\rb{\bar{j}} }<0,
\end{equation*}
Hence $\phi_{m,\nu}$ has an odd number of zeros on $\rb{j,\bar{j}}$.

To verify uniqueness, we suppose that $\phi_{m,\nu}$ has at least two zeros on $\rb{j,\bar{j}}$. Let $\omega$ and $\bar{\omega}$ be any consecutive positive zeros of $\phi_{m,\nu}$ on $\rb{j,\bar{j}}$. We observe that
\begin{align*}
    J_{\nu}\rb{\omega}\phi_{m,\nu}'\rb{\omega} &= W\qb{J_{\nu},\phi_{m,\nu}}(\omega) >0, \\ 
    J_{\nu}\rb{\bar{\omega}}\phi_{m,\nu}'\rb{\bar{\omega}} &= W\qb{J_{\nu},\phi_{m,\nu}}(\bar{\omega}) >0,
\end{align*}
which shows that $\phi_{m,\nu}'\rb{w}$ and $\phi_{m,\nu}'\rb{\bar{w}}$ have the same sign as long as $J_{\nu}$ maintains constant sign on $\rb{j,\bar{j}}$. By leveraging the simplicity of the zeros of $\phi_{m,\nu}$ and the continuity of $\phi_{m,\nu}$, it can be readily verified that there exists at least one zero of $\phi_{m,\nu}$ on $\rb{\omega,\bar{\omega}}$. This contradicts the choice of $\omega$ and $\bar{\omega}$. Consequently, whenever a pair of consecutive positive zeros of $J_{\nu}$ is selected, $\phi_{m,\nu}$ has one and only one zero between them.

In view of relations $\rho_{m-1,\nu,1}>j_{\nu,1}$ in \eqref{order} and $j_{\nu+m,1}>j_{\nu,1}$, we deduce that $\phi_{m,\nu}$ has no zeros on $\rb{0,j_{\nu,1}}$. Therefore the desired interlacing property follows.
\end{proof}

In the more general scenario where $\mathcal{N}_{m,\nu}$ is not an empty set, we can formulate the generalized interlacing property between $J_\nu$ and $J_{\nu+m}$, as detailed in the definition \ref{def:3.1}, in the following manner.
\begin{theorem}\label{Inter1}
For $\nu > -1$ and $m\in \mathbb{N}$, the positive zeros of $J_\nu$ and $R_{m-1,\nu+1}J_{\nu+m}$ are interlaced in the sense that
\begin{equation*}
    0 < {j}_{\nu,k_1} < \omega_{m,\nu,1} < {j}_{\nu,k_2} < \omega_{m,\nu,2} < \cdots.
\end{equation*}
\end{theorem}

\begin{proof}
Let us consider the auxiliary functions
\begin{equation*}
    U_{m,\nu}(x) = \frac{J_\nu(x)}{ p_{m,\nu}(x) }, \quad V_{m,\nu}(x) = \frac{R_{m-1,\nu+1}(x) J_{\nu+m}(x)}{ p_{m,\nu}(x) },
\end{equation*}
where $p_{m,\nu}$ is a monic polynomial of degree at most $\left\lfloor \frac{m-1}{2} \right\rfloor$, defined as
\begin{equation*}
    p_{m,\nu}(x) = \begin{cases}
        \prod_{\zeta \in \mathcal{N}_{m,\nu}} \rb{x- \zeta} &\ifz \mathcal{N}_{m,\nu} \ne \emptyset,\\
        1 &\ifz \mathcal{N}_{m,\nu} = \emptyset.
    \end{cases} 
\end{equation*}
note that $J_\nu$ has a countably infinite number of zeros which are all simple and real. (see \cite[\S 15]{Watson2})
We observe that
\begin{equation*}
    W\qb{ U_{m,\nu},\, V_{m,\nu} }(x) = \rb{p_{m,\nu}^{2}(x)}^{-1} W\qb{J_\nu,\, R_{m-1,\nu+1}J_{\nu+m}}(x),
\end{equation*}
which is an entire function with removable singularities at $x\in \mathcal{N}_{m,\nu}$. Moreover, it is strictly positive for $\nu>-1$, $m\in \mathbb{N}$ and $x >0$.

The exactly same line of reasoning as in the proof of Theorem \ref{thm:3.1} can be applied to the functions $U_{m,\nu}$ and $V_{m,\nu}$, in conjunction with the positivty of the Wronkskian $W\qb{ U_{m,\nu},\, V_{m,\nu} }$. This extends the established interlacing property.
\end{proof}

\begin{figure}[H]
    \centering
    \includegraphics[width=0.75\textwidth]{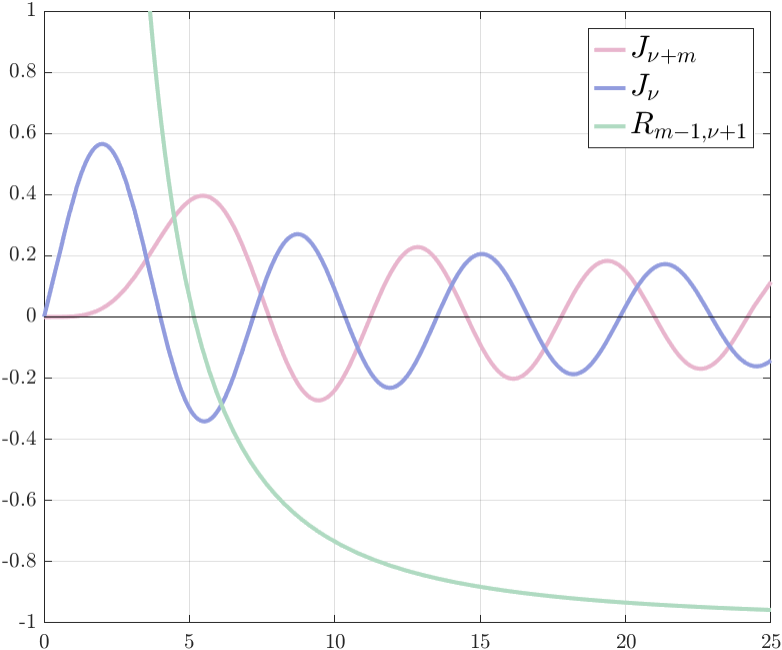}
    \caption{The graphs of $J_{\nu+m}$, $J_\nu$ and $R_{m-1,\nu+1}$ for $\nu=1.125$, $m=3$.}
    \label{fig:3.1}
\end{figure}

\begin{figure}[H]
    \centering
    \includegraphics[width=0.75\textwidth]{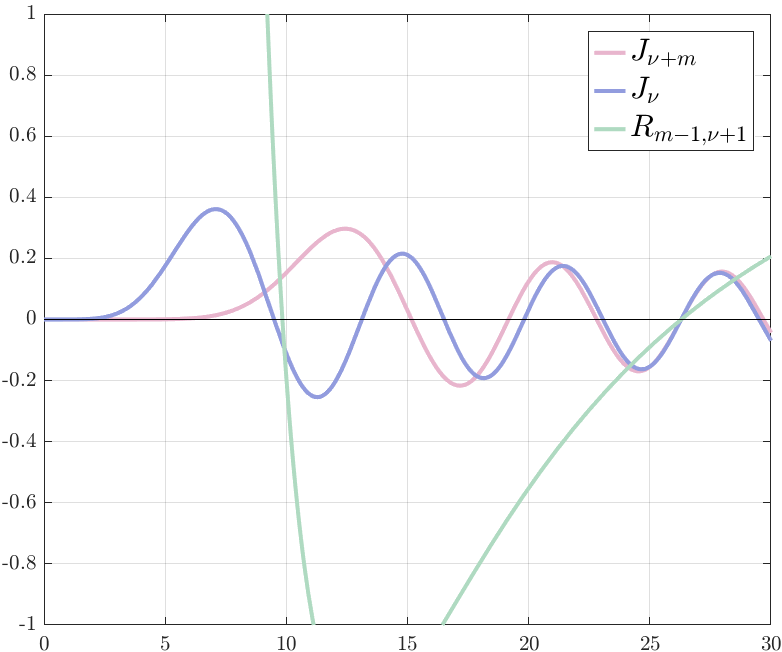}
    \caption{The graphs of $J_{\nu+m}$, $J_\nu$ and $R_{m-1,\nu+1}$ for $\nu=5.619\cdots$, $m=5$.}
    \label{fig:3.2}
\end{figure}

\begin{remark}\
\begin{enumerate}[label=(\roman*)]
    \item In particular, when $m=1$ and $m=2$, we have
    \begin{equation*}
        R_{0,\nu+1}(x)J_{\nu+1}(x) = J_{\nu+1}(x),\quad R_{1,\nu+1}(x)J_{\nu+2}(x) = \frac{2\rb{\nu+1}}{x} J_{\nu+2}(x),
    \end{equation*}
    which shows that Theorem \ref{thm:3.1} precisely reduces to the classical results, presented in \eqref{pattern}.

    \item It is impossible for $\mathcal{N}_{m,\nu}$ to contain two consecutive positive zeros of $J_\nu$. Suppose that $\omega_1,\,\omega_2\in \mathcal{N}_{m,\nu}$ with $\omega_1 = j_{\nu,\ell}$ and $\omega_2 = j_{\nu,\ell+1}$. Then it is apparent that $\omega_1,\,\omega_2 \in \cb{ \omega_{m,\nu,k} }_{k=1}^{\infty}$, and thus we have
    \begin{equation*}
        j_{\nu,\ell-1} < \omega_1 < \omega_2 < j_{\nu,\ell+2},
    \end{equation*}
    which contradicts to Theorem \ref{Inter1}.

    \item Figures \ref{fig:3.1} and \ref{fig:3.2} depict situations where the interlacing property between $J_\nu$ and $J_{\nu+m}$ is not valid, but the generalized interlacing property remains true. To illustrate, for a given positive zero $\rho$ of $R_{m-1,\nu+1}$, it follows
    \begin{equation*}
        j_{\nu+m,k} < j_{\nu,s} < \rho < j_{\nu,s+1} < j_{\nu+m,k+1},
    \end{equation*}
    for some $1\le k < s$. This scenario is illustrated in Figure \ref{fig:3.1} assuming $j_{\nu+m,0}=0$. On the other hand, when considering a given common zero $\zeta = j_{\nu,s}= j_{\nu+m,k}$ with fixed $1\le k < s$, we can observe that
    \begin{equation*}
        j_{\nu+m,k-1} < j_{\nu,s-1} < \zeta = j_{\nu,s} < j_{\nu,s+1} < j_{\nu+m,k+1},
    \end{equation*}
    which is depicted in Figure \ref{fig:3.2}. It's worth noting that the specific parameter $\nu$, as used in Figure \ref{fig:3.2}, is irrational number lying within the interval $\rb{5.619,5.62}$.

\end{enumerate}

\end{remark}

\section{Extension to Cylinder functions}

The generalized interlacing property can be established not only for the Bessel functions but also for the cylinder functions. We recall the definition of cylinder functions $C_\nu^\alpha= \cos{\alpha} J_\nu - \sin{\alpha} Y_\nu$, for real $\nu$ and $0\leq\alpha<\pi$. As the parameter $\alpha$ doesn't affect to the results, we abbreviate $C_\nu^\alpha$ by $C_\nu$. As customary, $\cb{ {c}_{\nu,k} }_{k=1}^{\infty}$ denotes the sequence of all positive zeros of $C_\nu$ arranged in ascending order of magnitude. Note that the cylinder functions also satisfy  \eqref{BL1} and \eqref{BL2}, which can be expressed as
\begin{align}
    &C_{\nu+m}(x)=R_{m,\nu}(x)C_\nu(x)-R_{m-1,\nu+1}(x)C_{\nu-1}(x), \label{CL1}\\
    &C_{\nu}(x)= - R_{m-2,\nu+1}(x) C_{\nu+m}(x)  + R_{m-1,\nu+1}(x) C_{\nu+m-1}(x) , \label{CL2}
\end{align}
and hence the equation \eqref{BL3} in the context of the cylinder function $C_\nu$ can be derived as follows:
\begin{equation}\label{C0}
    \begin{split}
        W\qb{ C_\nu,\, R_{m-1,\nu+1}C_{\nu+m} }(x)
    = -C_{\nu+m}^2(x) W\qb{ R_{m-2,\nu+1}, R_{m-1,\nu+1} }(x)\\
    + R_{m-1,\nu+1}^2(x) W\qb{ C_{\nu+m-1}, C_{\nu+m} }(x).
    \end{split}
\end{equation}

The prominent difference between the case of the Bessel functions $J_\nu$ and that of the cylinder functions $C_\nu$ is related to the positivity of the Wronskian. Specifically, $W[C_\nu,\, R_{m-1,\nu+1}C_{\nu+m}]$ doesn't maintain a constant sign over $(0,\infty)$. In fact, the Wronskian has at least one zero on $\rb{0,\infty}$. In such cases, additional analysis is necessary.

\begin{definition}
Let
\begin{align*}
    \cb{ {c}_{\nu,k_\ell} }_{\ell=1}^{\infty} &= \cb{ c_{\nu,k} }_{k=1}^{\infty} \setminus \mathcal{M}_{m,\nu},  \\
    \cb{ \tau_{m,\nu,k} }_{k=1}^{\infty} &= \cb{c_{\nu+m,k}}_{k=1}^\infty \cup \cb{ \rho_{m-1,\nu,k} }_{k\ge 1},
\end{align*}  
where  those are arranged in ascending order of magnitude and
\begin{equation*}
     \mathcal{M}_{m,\nu} = \cb{ c_{\nu,k}\,:\, C_{\nu+m}\rb{c_{\nu,k}}=0 },\quad m\in \mathbb{N}.
\end{equation*}
\end{definition}

\begin{theorem}
For $\nu>0$ and $m\in \mathbb{N}$, if $\mathcal{M}_{m,\nu}=\emptyset$, the positive zeros of $C_\nu$ and $R_{m-1,\nu+1}C_{\nu+m}$ are interlaced according to the following pattern
\begin{equation}\label{pattern2}
    0<c_{\nu,1}<\tau_{m,\nu,1}<c_{\nu,2}<\tau_{m,\nu,2}<\cdots.
\end{equation}
\end{theorem}

\begin{proof}
We shall investigate the interlacing pattern separately on the intervals $(0,c_{\nu+m,1})$ and $(c_{\nu+m,1},\infty)$. 

For the latter interval, we examine the Wronskian of cylinder functions just as in the previous section. A result due to P\'{a}lmai \cite[Lemma 5]{Palmai} provides, by considering $h^2W[f,g] = W[hf,hg]$, that for $\nu>0$,
\begin{equation}\label{C1}
    x\,W\qb{ C_{\nu+m-1}, C_{\nu+m} }(x)>0,\quad x> \min\rb{c_{\nu+m-1,1},c_{\nu+m,1}}.
\end{equation}
Obviously, we deduce that $\min\rb{c_{\nu+m-1,1},c_{\nu+m,1}} = c_{\nu+m-1,1}$ by \eqref{W}, and also that $W\qb{ R_{m-2,\nu+1}, R_{m-1,\nu+1} }(x)<0$, for $x>0$, by \eqref{BL5}. Therefore, on combining these with \eqref{C0} and \eqref{C1}, we obtain that for $\nu>0$,
\begin{equation*}
    W\qb{ C_\nu,\, R_{m-1,\nu+1}C_{\nu+m} }(x)> 0,\quad x>c_{\nu+m-1,1}.
\end{equation*}
By shifting $x$ by $x-c_{\nu+m-1,1}$ and applying the same argument as in the proof of Theorem \ref{thm:3.1}, it is straightforward to deduce that the positive zeros of $C_\nu$ and $R_{m-1,\nu+1}C_{\nu+m}$ are interlaced within the interval $\rb{c_{\nu+m-1,1},\infty}$, which certainly includes $(c_{\nu+m,1},\infty)$.

Secondly, we shift our focus to the interval $(0,c_{\nu+m,1})$. Since $c_{\nu,1}<c_{\nu+m,1}$, it is obvious that $C_\nu$ has at least one zero in the interval $(0,c_{\nu+m,1})$. If $R_{m-1,\nu+1}$ has no zeros in $(0,c_{\nu+m,1})$, then $C_\nu$ has only one zero in the same interval. This is because, in that case, $R_{m-1,\nu+1}$ and $C_{\nu+m}$ is positive on the interval and $C_{\nu-1}$ is positive at the possible second zero of $C_\nu$, which contradicts to \eqref{CL1} and hence establishes \eqref{pattern2}. 

Now suppose $R_{m-1,\nu+1}$ has at least one zero in $(0,c_{\nu+m,1})$. For the sake of convenience, let $c_1<c_2<\cdots<c_N$ and $\rho_1<\rho_2<\cdots<\rho_M$ be the all zeros of $C_\nu$ and $R_{m-1,\nu+1}$ within the interval $(0,c_{\nu+m,1})$. We claim that for each $k=1,\ld,N$ and $\ell=1,\ld,M$,
\begin{equation*}
    (-1)^{k+1} R_{m-1,\nu+1}(c_k) >0,\quad (-1)^\ell C_{\nu}(\rho_\ell) >0.
\end{equation*}
The first inequality comes from \eqref{CL1}, combined with the basic facts
\begin{equation*}
    (-1)^{k} C_{\nu-1}(c_k)>0,\quad k=1,\ld,N,
\end{equation*}
and
\begin{equation*}
    C_\nu(x)>0,\quad 0<x<c_{\nu,1}.
\end{equation*}
Similarly, the second inequality is obtained from
\begin{equation*}
    (-1)^\ell R_{m,\nu}(\rho_\ell)>0,\quad \ell=1,\ld,M.
\end{equation*}

Thus, in $(0,c_{\nu+m,1})$, the zeros of $C_\nu$ and those of $R_{m-1,\nu+1}$ are interlaced. Furthermore, the first zero of $C_\nu$ precedes that of $R_{m-1,\nu+1}$. Consequently, we are left with two possibilities:
\begin{align*}
    c_1&<\rho_1<c_2<\rho_2<\cdots<\rho_{N-1}<c_N, \tag{when $M=N-1$}\\
    c_1&<\rho_1<c_2<\rho_2<\cdots<\rho_{N-1}<c_N<\rho_N. \tag{when $M=N$}
\end{align*}
Suppose $M=N$. Firstly, the scenario where
\begin{equation*}
    c_{\nu+m-1,1}<c_N<\rho_N<c_{\nu+m,1}
\end{equation*}
is impossible, as it contradicts to the interlacing pattern proved above for $x>c_{\nu+m-1,1}$. For the same reason, the case of $c_N<c_{\nu+m-1,1}<\rho_N$ is also excluded. The only remaining case is $c_N<\rho_N<c_{\nu+m-1,1}$. In this instance, we consider the signs of \eqref{CL2} at $x=c_{\nu+m,1}$. The sign of the left-hand side is $(-1)^{N}$, while that of the right-hand side is $(-1)^{N+1}$, leading to a contradiction.

Therefore, we conclude that $M=N-1$, thus confirming the desired interlacing pattern and completing the proof.
\end{proof}

As in the case of Bessel functions of the first kind, when common zeros might occur, we can establish the following theorem, with a reasoning analogous to the proof of Theorem \ref{Inter1}.

\begin{theorem}
    For $\nu>0$ and $m\in \mathbb{N}$, the positive zeros of $C_\nu$ and\\
    $R_{m-1,\nu+1}C_{\nu+m}$ are interlaced in the sense that
\begin{equation*}
    0<c_{\nu,k_1}<\tau_{m,\nu,1}<c_{\nu,k_2}<\tau_{m,\nu,2}<\cdots.
\end{equation*}
\end{theorem}

\section{The derivative of the Bessel functions}

The reasoning employed in section 3 remains applicable in a broader and more general framework. For instance, an analogous version of the generalized interlacing property can be established by considering the derivative $J_\nu'$ of the Bessel function along with $J_{\nu+m}$. We begin by noting that
\begin{equation*}
        J_\nu'(x) = \frac12 \big( J_{\nu-1}(x) - J_{\nu+1}(x) \big).
\end{equation*}
By employing the equation \eqref{BL2} with $\nu \mapsto \nu-1+2k$ and $m \mapsto m+1-2k$ for each $k=1,\,2$, we derive the following recurrence relation
\begin{equation}\label{derivative1}
        J_\nu'(x) = -R^{*}_{m-1,\nu}(x) J_{\nu+m}(x) + R^{*}_{m,\nu}(x) J_{\nu+m-1}(x),
\end{equation}
where the associated polynomial $R^{*}_{m,\nu}(x)$ is defined as
\begin{align}\label{def*}
    \begin{split}
        R^{*}_{m,\nu}(x) &= \frac{1}{2}\big(R_{m,\nu}(x)-R_{m-2,\nu+2}(x)\big)  \\
        &=\frac{1}{2}\big(R_{m,\nu}(x)+(-1)^m R_{-m,-\nu}(x)\big).
    \end{split}
\end{align}
We present several examples of them as follows:
\begin{equation}\label{lommel*}
    R_{0,\nu}^{*}(x)=1,\quad R_{1,\nu}^{*}(x)=\frac{\nu}{x}, \quad  R_{2,\nu}^{*}(x)=\frac{2\nu(\nu+1)}{x^2}-1.
\end{equation}
It is apparent from the equation \eqref{def*} that for $\nu\in \R$ and $m\in \mathbb{Z}$,
\begin{equation}\label{Graf2}
    R_{m,\nu}^* (x) = (-1)^m R_{-m,-\nu}^*(x).
\end{equation}

Additionally, it can be readily verified that the associated polynomial $R^{*}_{m,\nu}$ satisfies the following recurrence relation
\begin{equation*}
    R^{*}_{m-1,\nu}(x) + R^{*}_{m+1,\nu}(x) = \frac{2(\nu+m)}{x} R^{*}_{m,\nu}(x),
\end{equation*}
which aligns with the recurrence \eqref{Lommel1}. Dividing both sides by $R_{m,\nu}^*$ and differentiating, we obtain
\begin{equation*}
    W[R_{m+1,\nu}^{*},R_{m,\nu}^{*}](x)=\frac{2(\nu+m)}{x^2}\big(R_{m,\nu}^{*}(x)\big)^2+ W[R_{m,\nu}^{*},R_{m-1,\nu}^{*}](x).
\end{equation*}
Continuing this iterative process, we have that for $m\in\mathbb{N}$,
\begin{equation}\label{AssoWron}
\begin{split}
    W[R_{m+1,\nu}^{*},&R_{m,\nu}^{*}](x)\\
    &=\frac{2(\nu+m)}{x^2}\big(R_{m,\nu}^{*}(x)\big)^2+\cdots+\frac{2(\nu+1)}{x^2}\big(R_{1,\nu}^{*}(x)\big)^2+\frac{\nu}{x^2},
\end{split}
\end{equation}
with an identity
\begin{equation*}
    W[R_{1,\nu}^{*},R_{0,\nu}^{*}](x)=\frac{\nu}{x^2}.
\end{equation*}

As in the proof of Proposition \ref{Wron1}, we deduce from the equations \eqref{derivative1} and \eqref{AssoWron} that for $m\in \mathbb{N}$,
\begin{align*}
    W\qb{J_\nu',\, R_{m,\nu}^{*}\,J_{\nu+m}}(x)
    = 2J_{\nu+m}^2(x) \left[ \big( R^{*}_{m,\nu}(x) \big)^2  \sum_{k=1}^\infty \frac{x^2+j_{\nu+m, k}^2}{\big( x^2 - j_{\nu+m, k}^2 \big)^2  } \right. \hspace{37.5pt} \\
    \left. + \frac{1}{x^2} \sum_{k=1}^{m} \rb{\nu+k} \big( R_{k,\nu}^{*}(x)\big)^2 +\frac{\nu}{2x^2}\right].
\end{align*}
If $\nu>0$ and $m\in \mathbb{N}$, the Wronskian vanishes at the common zeros of $J_{\nu+m}$ and $R_{m,\nu}^{*}$, and it is positive otherwise. This also holds true for $\nu=0$ and $m\ge 2$. In particular, in the case when $m=0$, a specific case arises as
\begin{equation*}
    W\qb{J_\nu',\, R_{0,\nu}^{*}\,J_{\nu}}(x) = W\qb{J_\nu',\, J_{\nu}}(x) = J_{\nu}^2 (x)\rb{ \frac{\nu}{x^2}+ 2\sum_{k=1}^\infty \frac{x^2+j_{\nu, k}^2}{\big( x^2 - j_{\nu, k}^2 \big)^2  }}.
\end{equation*}
On inspecting Wronskian $W\qb{J_\nu',\, R_{m,\nu}^{*}\,J_{\nu}}$, it becomes completely devoid of significance in cases where $R_{m,\nu}^*$ is identically zero for some $m$ and $\nu$.
\begin{lemma}\label{lem:5.1}
    Let $m\in \mathbb{N}\cup\cb{0}$ and $\nu >0$.
    The polynomial $R_{m,\nu}^*$ cannot be identically zero.
\end{lemma}

\begin{proof}
We first observe from the equation \eqref{lommel*} that $R_{m,\nu}^*$ is not identically zero if $m=0$ and $m=1$.

Now, we shall show that $R_{m,\nu}^*$ cannot be zero function for $m\ge 2$ and $\nu >0$. 
We examine the definition of Lommel polynomials in \eqref{L}, which yields that for each $m\ge 2$ and $\nu> 0$, as $x \to 0+$,
\begin{equation*}
    R_{m,\nu}(x) = (\nu)_m \rb{\frac{x}{2}}^{-m} + O\rb{x^{-m+2}},\quad R_{m-2,\nu+2}(x) = O\rb{x^{-m+2}}.
\end{equation*}
Moreover, we obtain that for each $m\ge2$ and $\nu > 0$,
\begin{equation*}
    R^{*}_{m,\nu}(x) = \frac{1}{2}\big(R_{m,\nu}(x)-R_{m-2,\nu+2}(x)\big) =\frac{(\nu)_m}{2} \rb{\frac{x}{2}}^{-m} + O\rb{x^{-m+2}}
\end{equation*}
as $x \to 0+$, which completes the proof.
\end{proof}

We denote by $\big\{ {j}_{\nu,k}' \big\}_{k=1}^{\infty}$ and $\big\{ {\rho}_{m,\nu,k}^* \big\}_{k\ge 1}$ the sequences of all positive zeros of $J_\nu'$ and $R_{m,\nu}^*$, respectively.
Let us introduce the extended notation
\begin{equation*}
    \mathcal{N}_{m,\nu}^* = \cb{ j_{\nu,k}'\,:\, J_{\nu+m}\big( j_{\nu,k}' \big) =0 }
    =  \cb{ j_{\nu,k}'\,:\, R^{*}_{m,\nu}\big( j_{\nu,k}' \big) =0 },
\end{equation*}
which represents the collection of positive zeros common to $J'_\nu$ and $J_{\nu+m}$. 
\begin{definition}
    Let
\begin{align*}
    \big\{ {j}_{\nu,k_\ell}' \big\}_{\ell=1}^{\infty} &= \big\{ j_{\nu,k}' \big\}_{k=1}^{\infty} \setminus \mathcal{N}_{m,\nu}^*,  \\
    \big\{ \omega_{m,\nu,k}^{*} \big\}_{k=1}^{\infty} &= \big\{j_{\nu+m,k}\big\}_{k=1}^\infty \cup \big\{ \rho_{m,\nu,k}^{*} \big\}_{k\ge 1},
\end{align*}
where those are arranged in ascending order of magnitude.
\end{definition}

Siegel and Shidlovskii (see \cite[\S 4]{Siegel} and \cite[p. 217]{Shidlovskii}) also proved that all the zeros of $J_\nu'(x)$, $x\ne0$, are transcendental. Thus an analogue of Observations \ref{ob1} and \ref{ob2} can be established as follows:
\begin{observation}
Let $\nu>0$ and $m\ge 0$. Then we have
\begin{enumerate}[label=\rm{(\roman*)}]
    \item $\vb{ \mathcal{N}_{m,\nu}^* } \le \left\lfloor \frac{m}{2} \right\rfloor$.
    \item $\mathcal{N}_{m,\nu}^*$ is an empty set if $\nu$ is rational.
\end{enumerate}
\end{observation}

In summary, with the aid of Lemma \ref{lem:5.1}, Wronskian $W\qb{J_\nu',\, R_{m,\nu}^{*}\,J_{\nu+m}}$ maintains a positive sign for $\nu>0$, $m\in \mathbb{N}\cup\cb{0}$ and $x\in \R$. Hence, akin to the reasoning expounded in the proofs of Theorems \ref{thm:3.1} and \ref{Inter1}, we can derive the following theorem.
\begin{theorem}\label{thm:5.1}
Let $\nu>0$ and $m\in \mathbb{N}\cup \cb{0}$. Then the positive zeros of $J_\nu'$ and $R_{m,\nu}^{*}\,J_{\nu+m}$ are interlaced in the sense that
\begin{equation*}
    0 < {j}_{\nu,k_1}' < \omega_{m,\nu,1}^{*} < {j}_{\nu,k_2}' < \omega_{m,\nu,2}^{*} < \cdots.
\end{equation*}
\end{theorem}

A simple consequence of Theorem \ref{thm:5.1}, with a crucial observation being the absence of positive zeros for the polynomial $R_{m,\nu}^{*}$, leads to the following:
\begin{corollary}
    Let $\nu>0$. Then the positive zeros of $J_\nu'$ and $J_{\nu+m}$ are interlaced according to the following pattern
    \begin{equation}\label{coro}
    0 < j_{\nu,1}' < j_{\nu+m,1} < j_{\nu,2}' < j_{\nu+m,2} < \cdots
    \end{equation}
    if $m=0$ or $m=1$.
\end{corollary}
This interlacing property between $J_\nu'$ and $J_{\nu+m}$ concurs with the classical result established by P{\'a}lmai and Apagyi \cite{Palmai2}, while interlacing property breaks down when $m=2$ and $\nu>0$, in the sense that
\begin{equation*}
        j_{\nu+2,k} < j_{\nu,s}' < \rho^* < j_{\nu,s+1}' < j_{\nu+2,k+1},
\end{equation*}
for some $1\le k < s$, where $\rho^*= \sqrt{2\nu(\nu+1)}$ represents the unique positive zero of $R_{2,\nu}^*(x) = 2\nu(\nu+1)/x^2 -1$ for $\nu>0$.

\bigskip \noindent
{\bf Acknowledgements.}
The first author is especially grateful to Mourad E. H. Ismail for insightful discussions.

\end{document}